\documentclass[12pt]{article}

\usepackage{amsmath,amsthm,amsfonts,amssymb,amscd,cite}
\usepackage{tikz}
\usepackage{color}

\newtheorem{theorem}{Theorem}[section]
\newtheorem{conjecture}[theorem]{Conjecture}
\newtheorem{corollary}[theorem]{Corollary}

\newtheorem{lemma}[theorem]{Lemma}

\newtheorem{proposition}[theorem]{Proposition}

\begin{document}

\title{On a Conjecture about Comparing the First and Second Zagreb Indices of Graphs}

\author{
Ali Ghalavand$^{a,}$\thanks{Corresponding author, email:\texttt{alighalavand@nankai.edu.cn}}}

\maketitle

\begin{center}
$^a$ Center for Combinatorics and LPMC, Nankai University, Tianjin 300071, China \\
\end{center}

\begin{abstract}
Let $G$ be a graph with order $n(G)$, size $m(G)$, first Zagreb index $M_1(G)$, and second Zagreb index $M_2(G)$. More than twenty years ago, it was conjectured that 
$\frac{M_1(G)}{n(G)} \leq \frac{M_2(G)}{m(G)}$. Later, Hansen and Vuki\v{c}evi\'c demonstrated that this conjecture does not hold for general graphs but is valid for chemical graphs. In this paper, as an extension of the study of chemical graphs, we investigate graphs in which the difference between the minimum and maximum degrees is at most $3$. We prove that any graph in this class that serves as a counterexample to the stated conjecture must have a minimum degree of $2$ and a maximum degree of $5$. Furthermore, we present infinitely many connected graphs that serve as counterexamples to this conjecture, all of which have a minimum degree of 2, a maximum degree of 5, and an order of at least 218.
\end{abstract}

\section{Introduction}
In this paper, we define the graph category to exclude non-simple graphs, graphs with isolated vertices, and the null graphs. For our purposes, let $ G $ represent a graph with a vertex set $ V(G) $ and an edge set $ E(G) $. The sizes of $ V(G) $ and $ E(G) $ are referred to as the order and size of $ G $, respectively. We denote the order of $ G $ as $ n(G) $ and its size as $ m(G) $. Let $ u $ be a vertex in $ G $. The open neighborhood of $ u $, which comprises the vertices adjacent to $ u $, is denoted by $ N_G(u) $. The cardinality of $ N_G(u) $ is called the degree of $ u $ and is denoted by $ d_G(u) $. We denote the maximum degree of a graph $G$ as $\Delta(G)$ and its minimum degree as $\delta(G)$. A graph with a maximum degree four or less is called a chemical graph. For two positive integers \(i\) and \(j\), the notation \(n_i(G)\) denotes the number of vertices in the graph \(G\) that have a degree of \(i\). Additionally, \(m_{i,j}(G)\) represents the number of edges in \(G\) that connect a vertex with degree \(i\) to a vertex with degree \(j\). Now, consider another graph $ G' $. The disjoint union of $ G $ and $ G' $ forms a new graph with a vertex set $ V(G) \cup V(G') $ and an edge set $ E(G) \cup E(G') $. We denote the disjoint union of $ G $ and $ G' $ as $ G \cup G' $. For a natural number $ l $, let $ K_l $ represent the complete graph on $ l $ vertices. For a natural number $ r $, the disjoint union of $ r $ complete graphs, each of the same order $ l $, is denoted as $ rK_l $. The notations $ S_l$, $P_l$, and $C_l$ represent star, path, and cycle graphs on $l$ vertices, respectively. 

Over fifty years ago, Gutman and Trinajsti\'c explored the relationship between total  $\pi$-electron energy and molecular structure. They identified two key terms that frequently appear in approximate expressions for total  $\pi$-electron energy:

\[ M_1(G) = \sum_{v \in V(G)} d_G(v)^2 \quad \text{and} \quad M_2(G) = \sum_{uv \in E(G)} d_G(u)d_G(v). \]

These terms, known as the first and second Zagreb indices, are among the oldest and most widely utilized descriptors of molecular structure based on vertex degree. For detailed theoretical insights into these indices, one can refer to a recent review \cite{BD1}, as well as studies such as \cite{z1, JZD1, AD1, HD0, HD1, XD1} and the references cited therein. More than twenty years ago, the following conjecture regarding the comparison of the first and second Zagreb indices was proposed.

\begin{conjecture}\label{conj1}{\rm\cite{con1-1,con1-2,con1-3}}
If $G$ is a connected graph, then 
\[\frac{M_1(G)}{n(G)}\leq\frac{M_2(G)}{m(G)}.\]
\end{conjecture}
In 2007, Hansen and Vuki\v{c}evi\'c \cite{con1-4} demonstrated that Conjecture \ref{conj1} does not hold for general graphs, but is valid for chemical graphs. They presented the graph $ S_6 \cup C_3 $ as a non-connected counterexample, as well as a connected counterexample with 46 vertices and 110 edges. Although this conjecture is disproven for both general connected and disconnected graphs, Vuki\v{c}evi\'c and Graovac \cite{con1-5} showed that it is true for trees. In fact, they proved that if $ G$ is a tree, then 
$\frac{M_1(G)}{n(G)} \leq \frac{M_2(G)}{m(G)}$ holds, with equality if and only if $ G \cong S_{n(G)}$. In \cite{con1-6}, Liu confirmed that this conjecture is true for unicyclic graphs. Specifically, he showed that if $G$ is a unicyclic graph, then 
$\frac{M_1(G)}{n(G)} \leq \frac{M_2(G)}{m(G)}$ holds, with equality if and only if $ G \cong C_{n(G)}$. Utilizing AutoGraphiX \cite{con1-1,con1-2,con1-3}, Caporossi, Hansen, and Vuki\v{c}evi\'c \cite{con1-7} found an infinite family of counterexamples with a cyclomatic number $\nu(\geq2)$. The cyclomatic number of a graph $G$ is defined as  $m(G) - n(G) + 1$. 

In this paper, we examine graphs where the difference between the minimum and maximum degrees is no greater than 3. Our findings indicate that any graph in this category that acts as a counterexample to Conjecture \ref{conj1} must have a minimum degree of 2 and a maximum degree of 5. Furthermore, we present infinitely many connected graphs that serve as counterexamples to this conjecture, all of which have a minimum degree of 2, a maximum degree of 5, and an order of at least 218.

\section{Main Results }
In this section, we present our main results. The following theorem states that Conjecture \ref{conj1} holds for any graph \( G \) with \( \Delta(G) - \delta(G) \leq 3 \) that satisfies either \( \delta(G) \neq 2 \) or \( \Delta(G) \neq 5 \).
To start, we define the set $\mathbb{F}(G) = \{\{d_G(u), d_G(v)\} : uv \in E(G)\}$.
\begin{theorem}\label{th1}
Let $G$ be a graph with $\Delta(G)-\delta(G)\leq3$. If either $\delta(G)\neq2$ or $\Delta(G)\neq5$ holds, then
\[\frac{M_1(G)}{n(G)}\leq\frac{M_2(G)}{m(G)}.\]
The equality holds if and only if $\mathbb{F}(G)$ is either $\{\{\Delta(G),\delta(G)\}\}$, $\{\{1,4\}$, $\{2,2\}\}$, or $\{\{3,6\},\{4,4\}\}$.
\end{theorem}

\begin{proof}
Consider a graph \( G \) such that \( \Delta(G) - \delta(G) \leq 3 \), and either \( \delta(G) \neq 2 \) or \( \Delta(G) \neq 5 \). Additionally, let $ \Theta(G) = M_2(G) - \frac{m(G)}{n(G)} M_1(G) $. To prove the theorem, it is sufficient to show that $ \Theta(G) \geq 0 $, with equality holding if and only if $ \mathbb{F}(G) $ is one of the following: $ \{\{\Delta(G), \delta(G)\}\} $, $ \{\{1, 4\}, \{2, 2\}\} $, or $ \{\{3, 6\}, \{4, 4\}\} $. To establish this, assume that $\sum_{\delta(G) \leq i \leq j \leq \delta(G) + 3} = \sum_{*}$. Since we are considering graphs with no isolated vertices, we have 
\begin{align*}
\sum_{\ast}m_{i,j}(G)(\frac{1}{i}+\frac{1}{j})&=\sum_{uv\in\,E(G)}(\frac{1}{d_G(u)}+\frac{1}{d_G(v)})\\
&=\sum_{u\in\,V(G)}\sum_{v\in\,N_G(u)}\frac{1}{d_G(u)}\\
&=\sum_{u\in\,V(G)}\frac{d_G(u)}{d_G(u)}=\sum_{u\in\,V(G)}1=n(G).
\end{align*}
So, based on the definitions of the first and second Zagreb indices, it is evident that,
\begin{align*}
  \Theta(G) & =\sum_{\ast}m_{i,j}(G)ij-\frac{\sum_{\ast}m_{i,j}(G)}{\sum_{\ast}m_{i,j}(G)(\frac{1}{i}+\frac{1}{j})}\sum_{\ast}m_{i,j}(G)(i+j).
\end{align*} 
Therefore, 
\begin{align*}
  \Theta(G) =&\frac{\sum_{\ast}m_{i,j}(G)ij\sum_{\ast}m_{i,j}(G)(\frac{1}{i}+\frac{1}{j})}{\sum_{\ast}m_{i,j}(G)(\frac{1}{i}+\frac{1}{j})}\\
  &-\frac{\sum_{\ast}m_{i,j}(G)\sum_{\ast}m_{i,j}(G)(i+j)}{\sum_{\ast}m_{i,j}(G)(\frac{1}{i}+\frac{1}{j})}.
\end{align*} 
Let's consider the function $\Theta'(G)$ defined as follows: \[\Theta'(G) = \left(\sum_{\ast} m_{i,j}(G) \left( \frac{1}{i} + \frac{1}{j} \right)\right) \Theta(G).\] It follows that $\Theta(G) \geq 0$ if and only if $\Theta'(G) \geq 0$, with equality holding under the same conditions. For the remainder of this proof, let $i$, $j$, $k$, and $l$ be four integers in the range from $\delta(G)$ to $\delta(G) + 3$. We define the function:
\[
\Psi(i,j,k,l) = ij\left(\frac{1}{k} + \frac{1}{l}\right) + kl\left(\frac{1}{i} + \frac{1}{j}\right) - (i + j + k + l).
\]
By utilizing the definition of $\Theta'(G)$, we can establish that proving the inequality $\Theta(G) \geq 0$ is equivalent to demonstrating that $\Psi(i,j,k,l) \geq 0$. Notably, the function $\Psi(i,j,k,l)$ possesses certain symmetric properties: specifically, it holds that $\Psi(i,j,k,l) = \Psi(j,i,k,l) = \Psi(i,j,l,k) = \Psi(k,l,i,j)$. Let $\delta = \delta(G)$. By employing the calculations for the function $\Psi$, as given in Table \ref{tab1}, we can deduce the following: $\Psi(i,j,k,l) = 0$ when $\{i,j\} = \{k,l\}$; $\Psi(i,j,k,l) = \frac{(\delta-1)(\delta-3)}{\delta(\delta+1)(\delta+3)}$ when $\{\{i,j\},\{k,l\}\} = \{\{\delta,\delta+3\},\{\delta+1,\delta+1\}\}$; and $\Psi(i,j,k,l) > 0$ in all other cases. 
Thus, since either $\delta(G)\neq2$ or $\Delta(G)\neq5$ holds, we can conclude that $\Psi(i,j,k,l) \geq 0$. The equality holds if and only if one of the following conditions is met: $\{i,j\} = \{k,l\}$, $\{\{i,j\},\{k,l\}\} = \{\{1,4\},\{2,2\}\}$, or $\{\{i,j\},\{k,l\}\} = \{\{3,6\},\{4,4\}\}$.

By combining the above arguments, we conclude that $\frac{M_1(G)}{n(G)} \leq \frac{M_2(G)}{m(G)}$. The equality holds if and only if $\mathbb{F}(G)$ is either $\{\{\Delta(G),\delta(G)\}\}$, $\{\{1,4\},$ $\{2,2\}\}$, or $\{\{3,6\},\{4,4\}\}$, as desired.
\end{proof}

\begin{table}[ht!]
  \begin{center}
   \caption{The calculations of the function $\Psi$ in the proof of Theorem \ref{th1}. To make the table smaller, we use the notation $\circ x$ instead of $\delta+x$, where $x\in\{i,j,k,l\}$.}
    \label{tab1}
    \begin{tabular}{|l|c|l|c|}
    \hline
      $(i,j,k,l)$ & $\Psi(\circ i,\circ j,\circ k,\circ l)$ &  $(i,j,k,l)$ & $\Psi(\circ i,\circ j,\circ k,\circ l)$ \\
      \hline
  $(0,0,0,0)$  &$0$ & $(0,0,0,1)$& $\frac{1}{\delta+1}$\\
   \hline
  $(0,0,0,2)$  &$\frac{4}{\delta+2}$ & $(0,0,0,3)$& $\frac{9}{\delta+3}$\\
   \hline
  $(0,0,1,1)$  &$\frac{2(2\delta+1)}{\delta(\delta+1)}$ & $(0,0,1,2)$& $\frac{(3\delta+2)(3\delta+4)}{\delta(\delta+1)(\delta+2)}$\\
   \hline
  $(0,0,1,3)$  &$\frac{2(4\delta+3)(2\delta+3)}{\delta(\delta+1)(\delta+3)}$& $(0,0,2,2)$& $\frac{16(\delta+1)}{\delta(\delta+2)}$\\
   \hline
   $(0,0,2,3)$  &$\frac{(5\delta+6)(5\delta+12)}{\delta(\delta+2)(\delta+3)}$& $(0,0,3,3)$& $\frac{18(2\delta+3)}{\delta(\delta+3)}$\\
   \hline
   $(0,1,0,1)$  &$0$ & $(0,1,0,2)$& $\frac{\delta}{(\delta+1)(\delta+2)}$\\
   \hline
   $(0,1,0,3)$  &$\frac{4\delta}{(\delta+1)(\delta+3)}$ & $(0,1,1,1)$& $\frac{1}{\delta}$\\
   \hline
   $(0,1,1,2)$  &$\frac{4(\delta+1)}{\delta(\delta+2)}$ & $(0,1,1,3)$& $\frac{9(\delta+1)}{\delta(\delta+3)}$ \\
   \hline
  $(0,1,2,2)$  &$\frac{(3\delta+2)(3\delta+4)}{\delta(\delta+1)(\delta+2)}$ & $(0,1,2,3)$& $\frac{4(2\delta+3)(2\delta^2+6\delta+3)}{\delta(\delta+1)(\delta+2)(\delta+3)}$ \\
   \hline
   $(0,1,3,3)$  &$\frac{(5\delta+3)(5\delta+9)}{\delta(\delta+1)(\delta+3)}$ & $(0,2,0,2)$& $0$  \\
   \hline
   $(0,2,0,3)$  &$\frac{\delta}{(\delta+2)(\delta+3)}$ & $(0,2,1,1)$& $\frac{2}{\delta(\delta+1)(\delta+2)}$\\
   \hline
   $(0,2,1,2)$  &$\frac{\delta+2}{\delta(\delta+1)}$ & $(0,2,1,3)$& $\frac{2(2\delta+3)(\delta^2+3\delta+3)}{\delta(\delta+1)(\delta+2)(\delta+3)}$\\
   \hline
  $(0,2,2,2)$  &$\frac{4}{\delta}$ & $(0,2,2,3)$& $\frac{9(\delta+2)}{\delta(\delta+3)}$\\
   \hline
   $(0,2,3,3)$  &$\frac{2(2\delta+3)(4\delta+9)}{\delta(\delta+2)(\delta+3)}$& $(0,3,0,3)$& $0$\\
   \hline
   $(0,3,1,1)$  &$\frac{(\delta-1)(\delta-3)}{\delta(\delta+1)(\delta+3)}$& $(0,3,1,2)$& $\frac{4(2\delta+3)}{\delta(\delta+1)(\delta+2)(\delta+3)}$\\
   \hline
   $(0,3,1,3)$  &$\frac{\delta+3}{\delta(\delta+1)}$& $(0,3,2,2)$& $\frac{(\delta+4)(\delta+6)}{\delta(\delta+2)(\delta+3)}$\\
   \hline
  $(0,3,2,3)$  &$\frac{4(\delta+3)}{\delta(\delta+2)}$& $(0,3,3,3)$& $\frac{9}{\delta}$\\
   \hline
  $(1,1,1,1)$  &$0$& $(1,1,1,2)$& $\frac{1}{\delta+2}$\\
   \hline
  $(1,1,1,3)$  &$\frac{4}{\delta+3}$& $(1,1,2,2)$& $\frac{2(2\delta+3)}{(\delta+1)(\delta+2)}$\\
   \hline
  $(1,1,2,3)$  &$\frac{(3\delta+5)(3\delta+7)}{(\delta+1)(\delta+2)(\delta+3)}$& $(1,1,3,3)$& $\frac{16(\delta+2)}{(\delta+1)(\delta+3)}$\\
   \hline
  $(1,2,1,2)$  &$0$& $(1,2,1,3)$& $\frac{\delta+1}{(\delta+2)(\delta+3)}$\\
   \hline
  $(1,2,2,2)$  &$\frac{1}{\delta+1}$& $(1,2,2,3)$& $\frac{4(\delta+2)}{(\delta+1)(\delta+3)}$\\
   \hline
 $(1,2,3,3)$  &$\frac{(3\delta+5)(3\delta+7)}{(\delta+1)(\delta+2)(\delta+3)}$& $(1,3,1,3)$& $0$\\
   \hline
  $(1,3,2,2)$  &$\frac{2}{(\delta+1)(\delta+2)(\delta+3)}$& $(1,3,2,3)$& $\frac{\delta+3}{(\delta+1)(\delta+2)}$\\
   \hline
  $(1,3,3,3)$  &$\frac{4}{\delta+1}$& $(2,2,2,2)$& $0$\\
   \hline
   $(2,2,2,3)$  &$\frac{1}{\delta+3}$& $(2,2,3,3)$& $\frac{2(2\delta+5)}{(\delta+2)(\delta+3)}$\\
   \hline
   $(2,3,2,3)$  &$0$& $(2,3,3,3)$& $\frac{1}{\delta+2}$\\
   \hline
  $(3,3,3,3)$  &$0$& & \\
   \hline
    \end{tabular}
  \end{center}
\end{table}

From the proof of the last theorem, we can conclude that if graph $G$ satisfies $\Delta(G) - \delta(G) \leq 3$ and serves as a counterexample for Conjecture \ref{conj1}, then $\delta(G) = 2$, $\Delta(G) = 5$, $m_{2,5}(G) \neq 0$, and $m_{3,3}(G) \neq 0$. Therefore, we can propose the following corollary.

\begin{corollary}
Let $G$ be a graph with $\Delta(G)-\delta(G)\leq3$. If $\frac{M_1(G)}{n(G)}>\frac{M_2(G)}{m(G)}$, then $\delta(G)=2$, $\Delta(G)=5$, $m_{2,5}(G)\neq0$, and $m_{3,3}(G)\neq0$.
\end{corollary}

In the following, we present infinitely many connected graphs that serve as counterexamples to Conjecture \ref{conj1}, all of which have a minimum degree of 2 and a maximum degree of 5. First, we begin with a lemma that will be useful for achieving our goal.
\begin{lemma}\label{lm1}
Let $ G_1 $ be a graph with a minimum degree of $ \delta(G_1) = 2 $ and a maximum degree of $ \Delta(G_1) = 5 $, such that $ m(G_1) = m_{2,5}(G_1) $. Additionally, let $ G_2 $ be a 3-regular graph. Consider a new graph $ G $ obtained by adding an edge between a vertex of degree $2$ in $ G_1 $ and a vertex of degree 3 in $ G_2 $. If $m(G_1)m(G_2)+360>81m(G_1)+80m(G_2)$, then 
 \[\frac{M_1(G)}{n(G)} > \frac{M_2(G)}{m(G)}.\]
\end{lemma} 
\begin{proof}
Suppose we define $\Psi(G) = m(G)M_1(G) - n(G)M_2(G)$. By examining the relationship between the structure of $G$ and the structures of $G_1$ and $G_2$, we can observe that
\begin{align*}
n(G) &= n(G_1) + n(G_2),\\
m(G)& = m(G_1) + m(G_2) + 1,\\
M_1(G) &= (m(G_1) - 2) \cdot 7 + (m(G_2) - 3) \cdot 6 + 44,\\
M_2(G) &= (m(G_1) - 2) \cdot 10 + (m(G_2) - 3) \cdot 9 + 78.
\end{align*}
From this, we conclude that
\begin{align*}
 \Psi(G)& =(7m(G_1)+6m(G_2)+12)(m(G_1)+m(G_2)+1)\\
  &-(10m(G_1)+9m(G_2)+31)(n(G_1)+n(G_2)).
\end{align*} 
On the other hand, we have the equations $n(G_1) = n_2(G_1) + n_5(G_1)$, $2m(G_1) = 2n_2(G_1) + 5n_5(G_1)$, and $2m(G_2) = 3n(G_2)$. From these, we find that
\[n(G_1) = m(G_1) - \frac{3}{2}n_5(G),~n(G_2) = \frac{2}{3} m(G_2).\]
Additionally, using the structure of $G_1$, we determine that $m(G_1) = 5n_5(G)$, which allows us to conclude that \[n(G_1) = \frac{7}{10}m(G_1).\] Therefore, we also conclude that 
\begin{align*}
 \Psi(G)& =(7m(G_1)+6m(G_2)+12)(m(G_1)+m(G_2)+1)\\
  &-(10m(G_1)+9m(G_2)+31)(\frac{7}{10}m(G_1)+\frac{2}{3} m(G_2))\\
  &=\frac{1}{30}m(G_1)m(G_2)+12-\frac{27}{10}m(G_1)-\frac{8}{3}m(G_2).
\end{align*} 
Now, given the inequality $m(G_1)m(G_2) + 360 \geq 81m(G_1) + 80m(G_2)$, it follows that $\Psi(G) > 0$, as desired.
\end{proof} 
For an integer \( k \geq 1 \), let \( P_{2k+1} \) and \( Q_{2k+1} \) be two disjoint paths of the same order \( 2k+1 \), defined as follows: 
\[
P_{2k+1} := v_1 v_2 \ldots v_{2k+1}, \quad Q_{2k+1} := u_1 u_2 \ldots u_{2k+1}.
\]
Additionally, for \( i \in \{3, 5, \ldots, 2k-1\} \), let \( H_i \) be a graph with vertex set \( V(H_i) = \{i_1, i_2, i_3\} \) and no edges. Moreover, we define \( H_1 \) and \( H_{2k+1} \) such that:
\[
V(H_1) = \{1_1, 1_2, 1_3, 1_4\}, \quad V(H_{2k+1}) = \{x_1, x_2, x_3, x_4\},
\]
where $x=2k+1$, and both edge sets are empty:
\[
E(H_1) = E(H_{2k+1}) = \emptyset.
\]
Now, consider a new graph \( \Xi_{2k+1} \) obtained from \( P_{2k+1} \), \( Q_{2k+1} \), and \( H_i \) for \( i \in \{1, 3, \ldots, 2k+1\} \) such that:
\begin{align*}
V(\Xi_{2k+1}) &= V(P_{2k+1}) \cup V(Q_{2k+1}) \cup V(H_1) \cup V(H_3) \cup \ldots \cup V(H_{2k+1}), \\
E(\Xi_{2k+1}) &= E(P_{2k+1}) \cup E(Q_{2k+1}) \cup \bigcup_{l \in L} \{v_l l_j, u_l l_j : j \in \{1, \ldots, 3\}\} \\
&\quad \cup \{v_i i_j, u_i i_j : i \in \{1, 2k+1\}, j \in \{1, \ldots, 4\}\},
\end{align*}
where \( L = \{3, 5, \ldots, 2k-1\} \). It is easy to see that \( \delta(\Xi_{2k+1}) = 2 \), \( \Delta(\Xi_{2k+1}) = 5 \), and \( m(\Xi_{2k+1}) = m_{2,5}(\Xi_{2k+1}) \), see Figure \ref{fig1} for more details.  Additionally, there are infinitely many 3-regular graphs. Based on these observations, and by employing Lemma \ref{lm1}, we can formulate the following proposition.

\begin{figure}[ht!]
\begin{center}
\begin{tikzpicture}
\clip(-7,-1) rectangle (-0.5,2.5);
\draw (-6.01,1.76)-- (-6,1.26);
\draw (-6,1.26)-- (-6,0.79);
\draw (-6,0.79)-- (-6.01,0.42);
\draw (-6.01,0.42)-- (-6,0);
\draw (-4.01,1.76)-- (-4,1.26);
\draw (-4,1.26)-- (-4,0.79);
\draw (-4,0.79)-- (-4.01,0.42);
\draw (-4.01,0.42)-- (-4,0);
\draw (-3.01,0.82)-- (-3,0.43);
\draw (-3,0.43)-- (-3,0);
\draw (-1.01,0.82)-- (-1,0.43);
\draw (-1,0.43)-- (-1,0);
\draw (-5.27,2.27)-- (-6.01,1.76);
\draw (-5.27,2.27)-- (-4.01,1.76);
\draw (-4.81,2.27)-- (-4.01,1.76);
\draw (-4.81,2.27)-- (-6.01,1.76);
\draw (-5.59,2.28)-- (-6.01,1.76);
\draw (-5.59,2.28)-- (-4.01,1.76);
\draw (-4.5,2.28)-- (-6.01,1.76);
\draw (-4.5,2.28)-- (-4.01,1.76);
\draw (-5,1.26)-- (-6,0.79);
\draw (-5,1.26)-- (-4,0.79);
\draw (-4.66,1.25)-- (-4,0.79);
\draw (-4.66,1.25)-- (-6,0.79);
\draw (-5.43,1.26)-- (-4,0.79);
\draw (-5.43,1.26)-- (-6,0.79);
\draw (-5.26,0.45)-- (-6,0);
\draw (-5.26,0.45)-- (-4,0);
\draw (-4.84,0.42)-- (-6,0);
\draw (-4.84,0.42)-- (-4,0);
\draw (-5.52,0.45)-- (-6,0);
\draw (-5.52,0.45)-- (-4,0);
\draw (-4.53,0.42)-- (-6,0);
\draw (-4.53,0.42)-- (-4,0);
\draw (-2.25,1.29)-- (-3.01,0.82);
\draw (-2.25,1.29)-- (-1.01,0.82);
\draw (-1.8,1.27)-- (-3.01,0.82);
\draw (-1.8,1.27)-- (-1.01,0.82);
\draw (-2.6,1.28)-- (-1.01,0.82);
\draw (-2.6,1.28)-- (-3.01,0.82);
\draw (-1.47,1.27)-- (-3.01,0.82);
\draw (-1.47,1.27)-- (-1.01,0.82);
\draw (-2.23,0.42)-- (-3,0);
\draw (-2.23,0.42)-- (-1,0);
\draw (-1.73,0.43)-- (-3,0);
\draw (-1.73,0.43)-- (-1,0);
\draw (-1.36,0.43)-- (-3,0);
\draw (-1.36,0.43)-- (-1,0);
\draw (-2.57,0.45)-- (-3,0);
\draw (-2.57,0.45)-- (-1,0);
\draw (-5.15,-0.06) node[anchor=north west] {$\Xi_{5}$};
\draw (-2.15,-0.04) node[anchor=north west] {$\Xi_{3}$};
\begin{scriptsize}
\fill [color=black] (-6.01,1.76) circle (1.5pt);
\fill [color=black] (-6,1.26) circle (1.5pt);
\fill [color=black] (-6,0.79) circle (1.5pt);
\fill [color=black] (-6.01,0.42) circle (1.5pt);
\fill [color=black] (-6,0) circle (1.5pt);
\fill [color=black] (-4.01,1.76) circle (1.5pt);
\fill [color=black] (-4,1.26) circle (1.5pt);
\fill [color=black] (-4,0.79) circle (1.5pt);
\fill [color=black] (-4.01,0.42) circle (1.5pt);
\fill [color=black] (-4,0) circle (1.5pt);
\fill [color=black] (-3.01,0.82) circle (1.5pt);
\fill [color=black] (-3,0.43) circle (1.5pt);
\fill [color=black] (-3,0) circle (1.5pt);
\fill [color=black] (-1.01,0.82) circle (1.5pt);
\fill [color=black] (-1,0.43) circle (1.5pt);
\fill [color=black] (-1,0) circle (1.5pt);
\fill [color=black] (-5.27,2.27) circle (1.5pt);
\fill [color=black] (-4.81,2.27) circle (1.5pt);
\fill [color=black] (-5.59,2.28) circle (1.5pt);
\fill [color=black] (-4.5,2.28) circle (1.5pt);
\fill [color=black] (-5,1.26) circle (1.5pt);
\fill [color=black] (-4.66,1.25) circle (1.5pt);
\fill [color=black] (-5.43,1.26) circle (1.5pt);
\fill [color=black] (-5.26,0.45) circle (1.5pt);
\fill [color=black] (-4.84,0.42) circle (1.5pt);
\fill [color=black] (-5.52,0.45) circle (1.5pt);
\fill [color=black] (-4.53,0.42) circle (1.5pt);
\fill [color=black] (-2.25,1.29) circle (1.5pt);
\fill [color=black] (-1.8,1.27) circle (1.5pt);
\fill [color=black] (-2.6,1.28) circle (1.5pt);
\fill [color=black] (-1.47,1.27) circle (1.5pt);
\fill [color=black] (-2.23,0.42) circle (1.5pt);
\fill [color=black] (-1.73,0.43) circle (1.5pt);
\fill [color=black] (-1.36,0.43) circle (1.5pt);
\fill [color=black] (-2.57,0.45) circle (1.5pt);
\end{scriptsize}
\end{tikzpicture}
\caption{ The graphs $\Xi_5$ and $\Xi_3$.}
\label{fig1}
\end{center}
\end{figure}
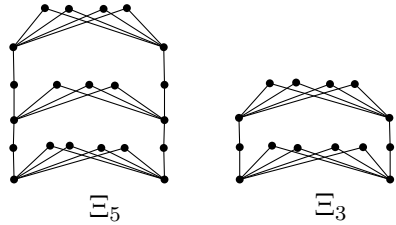

\begin{proposition}\label{pro1}
There are infinitely many connected graphs $G$ for which  $\delta(G) = 2$, $\Delta(G) = 5$, and the inequality $\frac{M_1(G)}{n(G)} > \frac{M_2(G)}{m(G)}$ holds true.
\end{proposition}

It is important to note that for any two positive integers $k$ and $l$, and any 3-regular graph \( H \), if we define \( G = \Xi_{2k+1} \cup H \) or $G= rK_{5,2} \cup lK_{4}$, we can use a method similar to the proof of Lemma \ref{lm1} to show that 
$\frac{M_1(G)}{n(G)} > \frac{M_2(G)}{m(G)}$. Consequently, there are infinitely many disconnected graphs \( G \) for which \( \delta(G) = 2 \), \( \Delta(G) = 5 \), and the inequality 
$\frac{M_1(G)}{n(G)} > \frac{M_2(G)}{m(G)}$ holds true.

In the next proposition, we demonstrate that the smallest connected graph created using the strategy outlined in Lemma \ref{lm1}, which serves as a counterexample for Conjecture \ref{conj1}, has an order of 218. 

\begin{proposition}\label{pro2}
Let \( k \) be a positive integer, and let \( H \) be a 3-regular graph. Additionally, let \( G \) be the graph obtained from \( \Xi_{2k+1} \) and \( H \) by adding an edge between a vertex of degree 2 in \( \Xi_{2k+1} \) and an arbitrary vertex in \( H \). If the inequality $\frac{M_1(G)}{n(G)} > \frac{M_2(G)}{m(G)}$ holds true, then it follows that \( n(G) \geq 218 \).
\end{proposition}
\begin{proof}
By utilizing the structures of \( H \) and \( \Xi_{2k+1} \), along with the statements from the proof of Lemma \(\ref{lm1}\), we can observe the following relationships:
\begin{equation}\label{eqq1}
m(H)=\frac{3}{2}n(H),~m(\Xi_{2k+1})=\frac{10}{7}n(\Xi_{2k+1})=10(k+1).
\end{equation}
Furthermore, by applying the conclusions from the proof of Lemma \(\ref{lm1}\) again, we can see that if the inequality 
\[
\frac{M_1(G)}{n(G)} > \frac{M_2(G)}{m(G)}
\]
holds, it leads to the following inequality:
\begin{equation}\label{eqq2}
m(\Xi_{2k+1})m(H)+360-81m(\Xi_{2k+1})-80m(H)>0.
\end{equation}
By combining Equations \eqref{eqq1} and \eqref{eqq2}, we derive \( k \geq 8 \) and 
\[
15 n(H) + 360 - 810(k + 1) - 120 n(H) > 0.
\]
This simplifies to:
\[
n(H) > \frac{810k + 450}{15k - 105}.
\]
Since \( n(\Xi_{2k+1}) = 7(k + 1) \) and \( n(G) = n(H) + n(\Xi_{2k+1}) \), it follows that 
\[
n(G) > \frac{(7k + 19)(k - 1)}{k - 7}.
\]
Next, for \( 8 \leq k \leq 15 \), by utilizing Equations \(\eqref{eqq1}\) and \(\eqref{eqq2}\), and noting that \( n(H) \) is an even integer, we find that \( n(G) \geq 218 \). For \( k \geq 16 \), since the function \( f(k) = \frac{(7k + 19)(k - 1)}{k - 7} \) is increasing on the interval \( (16, \infty) \), we conclude that \( n(G) > f(16) > 218 \).
Additionally, if we assume that \( H \) is a 3-regular graph of order 106, and let \( G \) be the graph obtained from \( \Xi_{31} \) and \( H \) by adding an edge between a vertex of degree 2 in \( \Xi_{31} \) and an arbitrary vertex in \( H \), we can confirm that \( n(G) = 218 \) and that 
\[
\frac{M_1(G)}{n(G)} > \frac{M_2(G)}{m(G)},
\]
as desired.

\end{proof}
\section{Concluding remarks}
The main contribution of this paper is an extension of the study related to Conjecture \ref{conj1} concerning chemical graphs. We examine graphs where the difference between the minimum and maximum degrees is at most 3. Our findings show that any graph in this category that acts as a counterexample to Conjecture \ref{conj1} must have a minimum degree of 2 and a maximum degree of 5. Moreover, we present infinitely many connected graphs that serve as counterexamples to this conjecture, all of which have a minimum degree of 2, a maximum degree of 5, and an order of at least 218. To further investigate these findings, we address two specific questions:\\
{\bf Question I:} What is the maximum integer \( n \) such that for any connected graph with a minimum degree of 2, a maximum degree of 5, and an order less than or equal to \( n \), Conjecture \ref{conj1} holds?\\
{\bf Question II:} How many distinct graphs with the same order exist that have a minimum degree of 2 and a maximum degree of 5, and which serve as counterexamples to Conjecture \ref{conj1}?




\begin{thebibliography}{99}
\bibitem{AD1} A. Ali, K.C. Das, S. Akhter, On the extremal graphs for second Zagreb index with fixed number of vertices and cyclomatic number, {\it Miskolc Mathematical Notes} {\bf 23} (2022) 41--50. 
    
\bibitem{con1-1}   M. Aouchiche, J. M. Bonnefoy, A. Fidahoussen, G. Caporossi, P. Hansen, P. L. Hiesse, J. Lacher\'e, A. Monhait, Variable neighborhood search for extremal graphs. 14. The AutoGraphiX 2 system, in: L. Liberti, N. Maculan (Eds.), {\it Global Optimization: From Theory to Implementation},Springer–Verlag, Berlin, 2005, pp. 281–310.

\bibitem{BD1}  B. Borovi\'canin, K.C.  Das, B.  Furtula, I.  Gutman,  Zagreb indices: Bounds and extremal graphs, {\it MATCH Commun. Math. Comput. Chem.} {\bf 78} (2017)  17--100. 

\bibitem{con1-2} G. Caporossi, P. Hansen, Variable neighborhood search for extremal graphs: 1 The AutoGraphiX system, {\it Discr. Math.} {\bf 212} (2000)  29--44.

\bibitem{con1-3} G. Caporossi, P. Hansen, Variable neighborhood search for extremal graphs. 5. Three ways to automate ﬁnding confectures, {\it Discr. Math.} {\bf 276} (2004) 81--94.

\bibitem{con1-7} G. Caporossi, P. Hansen, D. Vuki\v{c}evic, Comparing Zagreb indices of cyclic graphs, {\it MATCH Commun. Math. Comput. Chem.} {\bf63} (2010) 441--451.
    
\bibitem{con1-4} P. Hansen, D. Vuki\v{c}evi\'c, Comparing the Zagreb indices, {\it Croat. Chem. Acta} {\bf 80} (2007) 165--168.
    
\bibitem{HD0}  B. Horoldagva, K.C.  Das,  On Zagreb Indices of Graphs, {\it MATCH Commun. Math. Comput. Chem.} {\bf 85} (2021)  295--301. 
    
\bibitem{HD1} B. Horoldagva, K.C.  Das, T.-A.  Selenge,  Complete characterization of graphs for direct comparing Zagreb indices, {\it Discrete Appl. Math.} {\bf 215} (2016)  146--154. 

\bibitem{JZD1} Y. Jin, S. Zhou, T.  Tian, K.C.   Das,  Sufficient conditions for hamiltonian properties of graphs based on the difference of Zagreb indices, {\it Computational and Applied Mathematics} {\bf 43} (2024) Paper No. 385. 

\bibitem{con1-6} B. Liu, On a conjecture about comparing Zagreb indices, in: I. Gutman, B. Furtula (Eds.), Recent Results in the Theory of Randi\'c Index, {\it Univ. Kragujevac, Kragujevac},2008, pp. 205–209.

\bibitem{z1}  S. Nikoli\'c, G.  Kova\v{c}evi\'c, A.  Mili\'cevi\'c, N.  Trinajstić,  The Zagreb indices 30 years after, {\it Croat. Chem. Acta} \textbf{76} (2003)  113--124. 

\bibitem{con1-5} D. Vuki\v{c}evi\'c, A. Graovac, Comparing Zagreb $M_1$ and $M_2$ indices for acyclic molecules, {\it MATCH Commun. Math. Comput. Chem.} {\bf57} (2007) 587--590.
    
\bibitem{XD1} K. Xu, F.  Gao, K.C.  Das, N.  Trinajsti\'c,  A formula with its applications on the difference of Zagreb indices of graphs, {\it J. Math. Chem.} {\bf 57} (2019)  1618--1626. 



\end{thebibliography}
\end{document}